\documentclass{amsart}
\usepackage{cite}
\usepackage{amsmath}
\usepackage{amsthm}
\usepackage[all]{xy}
\usepackage{amssymb}
\usepackage{mathrsfs}
\usepackage{etoolbox}
\usepackage{fullpage}

\newtheorem{thm}{Theorem}
[section]
\newtheorem{cor}[thm]{Corollary}

\newtheorem{lem}[thm]{Lemma}
\newtheorem{prop}[thm]{Proposition}

\theoremstyle{definition}
\newtheorem{rem}[thm]{Remark}
\newtheorem{fct}[thm]{Fact}
\newtheorem{defn}[thm]{Definition}

\newtheorem{rmk}[thm]{Remark}

\numberwithin{equation}{subsection}

\def \ta {\tau_{\mathcal{D}/\Delta}}
\def \D {\Delta}
\def \DD {\mathcal D}

\newcommand{\gen}[1]{\left\langle#1\right\rangle}
\usepackage{color}
\begin{document}

\title[On parameterized differential Galois extensions]{On parameterized differential Galois extensions}

\author{Omar Le\'on S\'anchez}
\address{Omar Le\'on S\'anchez\\
McMaster University\\
Department of Mathematics and Statistics\\
1280 Main Street West\\
Hamilton, Ontario \  L8S 4L8\\
Canada}
\email{oleonsan@math.mcmaster.ca}

\author{Joel Nagloo}
\address{Joel Nagloo\\
Graduate Center\\
Mathematics\\
365 Fifth Avenue, New York\\
NY 10016-4309}
\thanks{Joel Nagloo was supported by NSF grant CCF-0952591.}
\email{jnagloo@gc.cuny.edu}

\date{\today}

\pagestyle{plain}

\begin{abstract}
We prove some existence results on parameterized strongly normal extensions for logarithmic equations. We generalize a result in [Wibmer, {\em Existence of $\partial$-parameterized Picard-Vessiot extensions over fields with algebraically closed constants}, J. Algebra, 361, 2012]. We also consider an extension of the results in [Kamensky and Pillay, {\em Interpretations and differential Galois extensions}, Preprint 2014] from the ODE case to the parameterized PDE case. More precisely, we show that if $\DD$ and $\D$ are two distinguished sets of derivations and $(K^{\DD},\D)$ is existentially closed in $(K,\D)$, where $K$ is a $\DD\cup\D$-field of characteristic zero, then every (parameterized) logarithmic equation over $K$ has a parameterized strongly normal extension. 

\end{abstract}

\maketitle

\section{Introduction}

Let $\Pi=\DD\cup\D=\{D_1,\dots,D_r\}\cup \{\delta_{1},\dots,\delta_{m-r}\}$ be a set of commuting derivations, with $m\geq r>0$, and $K$ be a $\Pi$-field of characteristic zero. Consider the (parameterized) system of homogeneous linear differential equations
\[D_1Y=A_1Y, \; \dots,\;  D_r Y=A_rY, \quad \text{ with $Y$ ranging in GL$_n$},\tag{$\star$}\]
where the $A_i$'s are  $n\times n$ matrices with entries from the differential field $K$ satisfying the usual integrability condition
$$D_i A_j -D_jA_i=[A_i,A_j],  \quad \text{ for } i,j=1,\dots,r.$$ 
Recall that a parameterized Picard-Vessiot (PPV) extension of $K$ for $(\star)$ is a $\Pi$-field extension $L$ of $K$ such that
\begin{enumerate}
\item $L$ is generated over $K$ by the entries (and all $\Pi$-derivatives) of a matrix solution $Z\in GL_n(L)$ of $(\star)$; in other words, $L=K\gen{Z}_\Pi=K\gen{Z}_\D$, and
\item $L^\DD=K^\DD$, that is the field of $\DD$-constants of $L$ is the same as the field of $\DD$-constants of $K$.
\end{enumerate}

For an example, take $K=(\mathbb{C}(x,t),\{\frac{\partial}{\partial x},\frac{\partial}{\partial t}\})$, where we think of $\DD$ as $\{\frac{\partial}{\partial x}\}$ and the parametric derivations $\D$ as $\{\frac{\partial}{\partial t}\}$. Let $(\star)$ be
\begin{equation}\label{useno}
\frac{\partial y}{\partial x}=\frac{t}{x}y.
\end{equation}
Clearly, $y=x^t$ is a solution and as $\frac{\partial}{\partial t}(x^t)= x^t\cdot log\:x$, one is interested in the field $L=K(x^t,log\:x)$. It turns out (see \cite[Example 3.1]{CaSi}) that $L$ is indeed a PPV extension of $K$ for equation (\ref{useno}).

Parameterized Picard-Vessiot extensions were introduced in \cite{CaSi} by Cassidy and Singer as a fundamental tool for studying parametric equations such as equation $(\star)$. In particular, they capture valuable information about the algebraic relations that exist among a set of solutions as well as their $\D$-algebraic relations (where $\D$ is the set of parametric derivations). PPV extensions have attracted much attention in recent years and we direct the reader to \cite{Mitschi} for some applications of the parameterized Picard-Vessiot theory. It should be noted that PPV extensions do not always exist; for example, consider the differential field $(\mathbb{R}\gen{\alpha},\frac{d}{dx})$, where $\alpha$ is a non constant solution of the equation $(\frac{dy}{dx})^2+4y^2+1=0$ in a differential closure of $(\mathbb R(x),\frac{d}{dx})$. Take $(\star)$ to be the linear differential equation
\[\frac{d^2y}{dx^2}+y=0.\] 
Seidenberg (see \cite[\S3]{Seidenberg}) proved that there are no Picard-Vessiot extensions of $K=\mathbb{R}\gen{\alpha}$ for this equation.

It has been known for quite some time that to get a general existence result, one needs to impose additional assumptions on $K^\DD$. In \cite{CaSi}, Cassidy and Singer showed that if $(K^\DD,\D)$ is $\D$-closed, then the existence of a PPV extension of $K$ is guaranteed. This was later improved in \cite{GilletAl} by Gillet et al.  where they show that the assumption can be weaken to $(K^\DD,\D)$ being existentially closed in $(K,\D)$. Examples of the latter occur for instance when $K$ is formally real and $(K^\DD,\D)$ is a real closed ordered differential field; more precisely, $(K^\DD, \D)$ is model of the theory $RCF\cup UC_{m-r}$ introduced by Tressl in \cite[\S8]{Tressl}. A similar observation can be made for the $p$-adic case now using the theory $pC_d\cup UC_{m-r}$. Another important result is due to Wibmer \cite{Wibmer}, who showed that in the case of a single parameter (i.e., $\D=\{\delta\}$), a PPV extension exists whenever $K^\DD$ is algebraically closed.

The most recent result can be found in \cite{MoshePillay} and concerns the nonlinear case with no parametric derivations (i.e, $\D=\emptyset$). In that paper, using model theoretic techniques, Kamensky and Pillay give a new proof of the nonparametric part of the above mentioned result of Gillet et al. \cite{GilletAl}. Moreover, their results are presented in the general context of logarithmic equations and strongly normal extensions (a generalization of Picard-Vessiot extensions to the nonlinear setting).

The aim of this current paper is to extend the work of Kamensky and Pillay to the parametric case. We work in the more general context of parameterized logarithmic equations (not necessarily linear) and parameterized strongly normal equations. In the process we obtain a new proof of the result of Gillet et al. about the existence of PPV extensions when $(K^\DD,\D)$ is existentially closed in $(K,\D)$. We also discuss ways in which one can extend the main result of Wibmer \cite{Wibmer} when there are more than one parameters (i.e., $|\D|>1$).

The paper is organized as follows. In Section 2, we give a quick review of the notion of parameterized D-varieties and D-groups. We then, in Section 3, explain what we mean by parameterized logarithmic equations and parameterized strongly normal (PSN) extensions and study their basic properties such as their Galois group of $\Pi$-automorphisms and the Galois correspondence. Section 4 and 5 is where our main results are proved. We first develop a quantifier elimination result for ``parameterized D-groups" and then use the latter to prove our main result about existence of PSN extensions. 


\subsection*{Acknowledgements}
The second author would like to thank the Logic group at McMaster University for their generous hospitality during his stay in December 2014 when part of the work presented here was completed.


\section{Preliminaries on parameterized D-varieties and D-groups}

In this section we recall the notions of parameterized prolongations of differential algebraic varieties and parameterized D-varieties with respect to a fixed partition of the distinguished derivations. We refer the reader to \cite[\S3]{Omar1} for a more detailed discussion (there the terminology \emph{relative} is used instead of parameterized). 

Throughout $(\mathcal{U},\Pi)$ will denote a sufficiently large saturated model of $DCF_{0,m}$; that is, $(\mathcal U,\Pi)$ will be play the role of our universal differential field (of characteristic zero) with $m$ commuting derivations for differential algebraic geometry. We also fix a ground (small) $\Pi$-subfield $K<\mathcal{U}$.  We consider a partition $\mathcal{D}\cup\Delta$ of $\Pi$, where $\mathcal{D}=\{D_1,\ldots,D_r\}$, $r>0$, and $\Delta=\{\delta_1,\ldots,\delta_{m-r}\}$. For any $\Pi$-subfield $F\leq \mathcal U$ we let 
$$F^{\mathcal{D}}:=\{a\in F: Da=0 \text{ for all } D\in \mathcal D\}$$
be the $\mathcal{D}$-constants of $F$. Recall that, by quantifier elimination of $DCF_{0,m}$, any $F$-definable set is a finite boolean combination of $\Pi$-algebraic varieties ($\subseteq\mathcal U^n$ for some $n$) over $F$. 

An important fact about the $\mathcal{D}$-constants that we will use several times in this paper is the following:
\begin{fct}\label{ConstFact} \
\begin{enumerate}
\item $(\mathcal{U}^{\mathcal{D}},\D)$ is a differentially closed $\Delta$-field.
\item Suppose $F$ is a  $\Pi$-subfield of $\mathcal U$. The $F$-definable subsets of Cartesian powers of $\mathcal{U}^{\mathcal{D}}$ in the structure $(\mathcal U,\Pi)$ are precisely the $F^{\mathcal{D}}$-definable sets in the structure $(\mathcal{U}^{\mathcal{D}},\Delta)$.
\end{enumerate}
\end{fct}

Now, let $x=(x_1,\dots,x_n)$ and $u=(u_1,\dots,u_n)$ be $n$-tuples of $\Delta$-indeterminates. We denote by $K\{x\}_\D$ the $\D$-ring of $\D$-polynomials over $K$ and by $\Theta_{\Delta}$ the set of $\Delta$-derivatives; that is, 
$$\Theta_{\Delta}=\{\delta_1^{e_1}\cdots \delta_{m-r}^{e_{m-r}}:\, e_i\geq0\}.$$
For each $D\in\DD$ and $f\in K\{x\}_{\Delta}$ we have a $\Delta$-polynomial $d_{D/\Delta}f\in K\{x,u\}_{\Delta}$ given by
\[d_{D/\Delta}f(x,u)=\sum_{\theta\in\Theta_{\Delta}, j\leq n}\frac{\partial f}{\partial (\theta x_j)}(x)\theta u_j+f^{D}(x),\]
where the $\Delta$-polynomial $f^{D}$ is obtained by applying $D$ to the coefficients of $f$. 

\begin{defn} The parameterized prolongation, $\tau_{\mathcal{D}/\Delta}V\subseteq \mathcal U^{n(r+1)}$, of an affine $\D$-algebraic variety $V\subseteq \mathcal U^n$ over $K$ is defined as the fibered product
\[\tau_{\mathcal{D}/\Delta}V=\tau_{D_1/\Delta}V\times_V\cdots\times_V\tau_{D_r/\Delta}V\]
where $\tau_{D_i/\Delta}V\subseteq U^{2n}$, $i=1,\dots,r$, is the affine $\Delta$-algebraic variety defined by 
\[f(x)=0 \text{ and } d_{D_i/\Delta}f(x,y)=0,\]
for all $f\in I_\Delta(V/K):=\{f\in K\{x\}_\Delta: f(V)=0\}$. We equip $\ta V$ with its canonical projection $\pi:\ta V\to V$ to the $x$-coordinate. 
\end{defn}

More generally, for an arbitrary $\D$-algebraic variety (i.e., not necessarily affine), the \emph{parameterized prolongation} $\ta V$ is defined by piecing together the prolongations of a (finite) affine cover of $V$ (see \cite[\S2.3]{Omar2} for details). For the basic properties of $\tau_{\mathcal{D}/\Delta}$ we refer the reader to \cite[\S3]{Omar1}. For instance, $\ta$ is a (covariant) functor from the category of $\D$-algebraic varieties to itself, it preserves differential fields of definition, and has the characteristic property that if $v\in V$ then $\nabla_\DD v:=(v,D_1v,\dots,D_rv)\in \ta V$. Moreover, for each $D\in\DD$, $\pi:\tau_{D/\D} V\to V$ is a torsor under the $\Delta$-tangent bundle $\rho:T_\Delta V\to V$ (where the fibrewise action is translation), and if $V$ is defined over $K^{D}$ then $\tau_{D/\D} V=T_\Delta V$.

\begin{rem}
When $\mathcal D=\{D\}$ and $\Delta=\emptyset$, the parameterized prolongation $\ta V$ is nothing more than the usual prolongation $\tau V$ used in ordinary differential algebraic geometry \cite{PiPi}. In this case, whenever $V$ is defined over the constants we recover the tangent bundle of $V$.
\end{rem}

By a \emph{parameterized D-variety} defined over $K$ we mean a pair $(V,s)$ where $V$ is a $\Delta$-algebraic variety and $s$ is a $\Delta$-section of $\tau_{\mathcal{D}/\Delta}V\rightarrow V$ (both defined over $K$) satisfying the following integrability condition: for each $v\in V$,
\[d_{D_i/\Delta}s_j(v,s_i(v))= d_{D_j/\Delta}s_i(v,v_j(x)), \: \text{ for } i,j=1,\ldots,r,\]
where $s = (\operatorname{Id},s_1,...,s_r)$ are local coordinates for $s$ in an affine chart containing the point $v$. 

By a \emph{parameterized $D$-subvariety} of $(V,s)$ we mean a $\Delta$-subvariety $W$ of $V$ such that $s(W)\subset \ta W$. A D-morphism of parameterized D-varieties $(V,s)$ and $(V',s')$ is a $\Delta$-morphism $f:V\to W$ such the following diagram commutes
$$\xymatrix{
\ta V \ar[rr]^{\ta f}&&\ta V'\\
V \ar[u]^{s}\ar[rr]^{f}&&V'\ar[u]_{s'}
}$$
This yields a category of parameterized D-varieties with D-morphisms. Since $\ta$ commutes with products, this category has products and thus, given a $\Delta$-algebraic group $G$, the parameterized prolongation $\ta G$ also has the structure of a $\Delta$-algebraic group (see \cite[\S4]{Omar1}). Hence, we can talk about group objects in this category. These are called \emph{parameterized D-groups} and they are precisely those parameterized D-varieties where the underlying $\Delta$-variety is a $\Delta$-algebraic group and the section is a group homomorphism. A \emph{parameterized D-subgroup} is defined in the natural way. 

The set of \emph{sharp points} of $(V,s)$, denoted by $(V,s)^\sharp$ or simply $V^\sharp$ when $s$ is understood, is the $\Pi$-algebraic subvariety of $V$ given by
\[ V^{\sharp}=\{v\in V:s(v)=\nabla_\DD v\}.\]
For an arbitrary subset $A\subseteq V$, we let $A^\sharp:=A\cap V^\sharp$.

\begin{lem} \label{onsubvar}\
\begin{enumerate}
\item Every irreducible $\Delta$-component of a parameterized D-variety is a parameterized D-subvariety.
\item A $\Delta$-subvariety $W$ of a parameterized D-variety is a parameterized D-subvariety if and only if $W^\sharp$ is $\Delta$-dense in $W$. Moreover, the $\sharp$-points functor establishes a 1:1 correspondence between parameterized D-subvarieties of V and $\Pi$-algebraic subvarieties of $V^\sharp$.
\item Intersections and finite unions of parameterized D-subvarieties are again parameterized D-subvarieties. 

\end{enumerate}
\end{lem}
\begin{proof} Fix a parameterized D-variety $(V,s)$.
\begin{enumerate}
\item If $W$ is a $\Delta$-component of $V$, then $\ta W$ and $\ta V$ agree on a nonempty $\Delta$-open subset $Q$ of $W$ (see \cite[Lemma 2.3.9]{Omar2}). Then, $s(Q)\subset \ta W$, but since $Q$ is $\Delta$-dense in $W$, we have $s(W)\subset \ta W$.
\item This is \cite[Proposition 3.10]{Omar1}.
\item By Noetherianity of the $\Delta$-topology (and induction) it suffices to show that the intersection of two parameterized D-subvarieties $W_1$ and $W_2$ is again a parameterized D-subvariety. Moreover, it suffices to consider the affine case. By (2) we must show that $(W_1\cap W_2)^\sharp$ is $\Delta$-dense in $W_1\cap W_2$, but this follows from the next equalities of ideals:
$$I_\Delta(W_1^\sharp\cap W_2^\sharp)=\sqrt{I_\D(W_1^\sharp)+I_\D(W_2^\sharp)}=I_\Delta(W_1\cap W_2),$$
where the last equality uses $I_\Delta(W_i^\sharp)=I_\Delta(W_i)$, for $i=1,2$, which holds by (2). To see that the union $W_1\cup W_2$ is a parameterized D-subvariety, we use (2) and
$$\operatorname{\Delta-Clo}((W_1\cup W_2)^\sharp)=\operatorname{\Delta-Clo}(W_1^\sharp)\cup\operatorname{\Delta-Clo(W_2^\sharp)}=W_1\cup W_2,$$
where $\operatorname{\Delta-Clo}$ denotes closure in the $\Delta$-topology.
\end{enumerate}
\end{proof}

The following lemma and proposition will be used in Section \ref{QED}. 

\begin{lem}\label{usil}
Let $(V,s)$ be a parameterized D-variety defined over $K$. Suppose $A$ is a (finite) Boolean combination of parameterized D-subvarieties (which are not necessarily defined over $K$). If A is definable over $K$, then $A$ is a Boolean combination of parameterized D-subvarieties each defined over $K$.
\end{lem}
\begin{proof}
This follows by induction on the rank (Morley rank say) of $A$. Let $\bar A=\D$-$\operatorname{Clo}(A)$ be the closure of $A$ in the $\D$-topology. We note that $\bar A$ is also defined over $K$. Then, by Lemma \ref{onsubvar}, $\bar A$ is a parameterized D-subvariety of $V$ defined over $K$. Thus, the set $\bar A\setminus A$ is definable over $K$ and is a Boolean combination of parameterized D-subvarieties. By induction, $\bar A\setminus A$ is a Boolean combination of parameterized D-subvarieties each defined over $K$. The result now follows since $A=\bar A\setminus (\bar A\setminus A)$.
\end{proof}

The following is  the parameterized PDE analog of \cite[Lemma 2.5(ii)]{KoPi}.

\begin{prop}\label{uselater}
Let $(V,s)$ be a parameterized D-variety. If A is a (finite) Boolean combination of parameterized D-subvarieties of $V$, then $A^\sharp$ is $\D$-dense in $A$.
\end{prop}
\begin{proof}
By Lemma \ref{onsubvar} (3), we may assume that $A$ is of the form $W\cap Q$ where $W$ is a parameterized D-subvariety of $V$ and $Q$ is a $\D$-open set. Moreover, by (1) of the same lemma, we may assume that $W$ is irreducible. Now, by (2) of Lemma \ref{onsubvar}, we have that $V^\sharp$ is $\Delta$-dense in $V$, and hence $A^\sharp$ is $\Delta$-dense in $A$.
\end{proof}

We finish this section with the following proposition which establishes the connection between parameterized D-groups and definable groups in $(\mathcal U,\Pi)$ of $\Pi$-type bounded by $\Delta$. Recall that given a $K$-definable set $X$ of $(\mathcal U,\Pi)$ we say that $\D$ \emph{bounds} the $\Pi$-type of $X$ if for each $a\in X$ the $\Pi$-field generated by $a$ over $K$ is also finitely generated as a $\Delta$-field, i.e., there exists a (finite) tuple $\alpha$ such that $K\langle a\rangle_\Pi=K\langle\alpha\rangle_\Delta$.

\begin{prop}
The $\sharp$-points functor is an equivalence between the category of parameterized D-groups over $K$ and the category of $\Pi$-algebraic groups over $K$ with $\Pi$-type bounded by $\Delta$. 
\end{prop}
\begin{proof}
By \cite[Theorem 4.6]{Omar1}, every $\Pi$-algebraic group of $\Pi$-type bounded by $\Delta$ is isomorphic to $G^\sharp$ for some parameterized D-group $(G,s)$. On the other hand, given two parameterized D-groups $(G,s)$ and $(H,t)$, any $\Pi$-isomorphism between $G^\sharp$ and $H^\sharp$ extends to a \emph{generically} defined $\Delta$-morphism of the $\Delta$-algebraic groups $G$ and $H$. It is well known that every such generically defined $\Delta$-morphism extends to an actual $\Delta$-isomorphism. The result follows. 
\end{proof}

\section{On parameterized strongly normal extensions}

We use the notation of the previous section; in particular, we have a partition $\mathcal D\cup\Delta$ of the set of derivations $\Pi$ where $\mathcal{D}=\{D_1,\ldots,D_r\}$, $r>0$, and $\Delta=\{\delta_1,\ldots,\delta_{m-r}\}$. Throughout this section we fix a $\Delta$-algebraic group $G$ defined over $K^{\mathcal D}$ and denote its idenity element by $e$. Note that in this case 
$$\ta G=(T_\Delta G)^r:=T_\Delta G\times_G\cdots \times_G T_\Delta G,$$
and thus the fibre of $\ta G$ over the identity $e$ of $G$ is equal to the $r$-fold product of $\mathfrak L_\Delta G$ (the $\Delta$-Lie algebra of $G$). Furthermore, we equip $G$ with the canonical \emph{zero section} $s_0$ of $\ta G$ and, henceforth, we will work with the parameterized D-group $(G,s_0)$. We note that the parameterized D-subgroups of $(G,s_0)$ defined over $K$ are precisely the $\Delta$-algebraic subgroups of $G$ that are defined over $K^\DD$.

A point $a\in \ta G_e$ is said to be \emph{integrable} if $(G,as_0)$ is a parameterized D-variety. Here $as_0:G\to \ta G$ denotes the $\D$-section given by $(as_0)(y)=a\cdot s_0(y)$ where the product occurs in $\ta G$. Since $\nabla_{\DD} e$ is the identity of $\ta G$, the point $\nabla_\DD e$ is integrable.

\begin{rem}
If $G$ is a linear (i.e., a $\D$-algebraic subgroup of some GL$_n$), then a point $(\operatorname{Id},A_1,\dots,A_r)\in \ta G_e\subseteq \mathfrak{gl}_n^r$ is integrable if and only if
$$D_i A_j -D_jA_i=[A_i,A_j],  \quad \text{ for } i,j=1,\dots,r.$$
These are usual integrability conditions for homogeneous linear differential algebraic equations, see \cite{CaSi} for instance.
\end{rem}

We have the following:

\begin{fct}\label{intpts}\cite[Lemma 5.3]{Omar1}
A point $a\in \ta G_e$ is integrable if and only if the system of $\Pi$-algebraic equations
$$\nabla_\DD y=a\cdot s_0(y)$$
has a solution in $G$.
\end{fct}

We denote the set of integrable point of $(G,s_0)$ by $\operatorname{Int}(G,s_0)$. The {\em parameterized logarithmic derivative} on $(G,s_0)$ is defined by
\begin{eqnarray*}
\ell_{0}: G & \rightarrow & \quad\; \ta G_e\\
 g & \mapsto & \nabla_{\mathcal D}(g)\cdot s_0(g)^{-1}
\end{eqnarray*}
where the product and inverse occur in the $\Delta$-algebraic group $\ta G$. By Fact~\ref{intpts}, $\operatorname{Int}(G,s_0)$ is the image of $\ell_0$. Moreover, $\operatorname{Ker}(\ell_0)=(G,s_0)^\#=G(\mathcal U^\DD)$, and $\ell_0$ is a crossed homomorphism with respect to the adjoint action of $G$ on $\ta G_e$ given by $a^g:=\ta C_g(a)$ where $C_g$ denotes conjugation by $g$ in $G$ (see \cite[Lemma 5.4]{Omar1}).

For the remainder of this section we fix an integrable $K$-point $a$ of $(G,s_0)$ and a \emph{parameterized logarithmic equation}
\[\ell_{0}(y)=a, \quad \text{ where } y \text{ ranges in } G \tag{$\star$}.\]

\begin{defn} A {\em parameterized strongly normal} (PSN) extension of $K$ for $(\star)$ is a $\Pi$-field $L$ such that
\begin{enumerate}
\item $L=K\gen{\alpha}_{\Delta}$, for $\alpha$ a solution of $(\star)$, and
\item $L^{\mathcal{D}}=K^{\mathcal{D}}$
\end{enumerate}
\end{defn}

\begin{rmk}\label{remonex} \
\begin{enumerate}
\item As explained in Example 5.6 of \cite{Omar1}, if $G=$ GL$_n$, then the parameterized logarithmic equation $(\star)$ reduces to the homogeneous linear differential system 
$$D_1 Y=A_1 Y, \; \dots \; , D_r Y=A_r Y,$$
where the integrable $K$-point has the form $a=(\operatorname{Id},A_1,\dots,A_r)\in \mathfrak{gl}_n^r(K)$. Moreover, in this case, a parameterized strongly normal extension for $(\star)$ is precisely a parameterized Picard-Vessiot extension for the above linear system.
\item Suppose $L=\gen {\alpha}_\D$ is a PSN extension of $K$ for $(\star)$. If $\sigma:L\to \mathcal U$ is a $\Pi$-isomorphism over $K$, then $\sigma(\alpha)=\alpha\cdot c$ for some $c\in G(\mathcal U^\DD)$, which implies that $\sigma(L)\subset L\gen{\mathcal U^\DD}_\D$. Thus, PSN extensions for $(\star)$ are examples of the \emph{$\D$-strongly normal extensions} studied by Landesman in \cite{Land}.
\end{enumerate}
\end{rmk}

As a consequence of the definition, PSN extensions are contained in a differential closure of $(K,\Pi)$:

\begin{lem}\label{ongalgro}
Suppose $L=K\gen{\alpha}_\D$ is a PSN extension of $K$ for $(\star)$. Then, $L$ is a contained in a $\Pi$-closure of $K$. Consequently, every $\Pi$-isomorphism $\sigma:L\to \mathcal U$ over $K$ extends uniquely to a $\Pi$-automorphism of $L\gen{\mathcal U^\DD}_\D$ over $K\gen{\mathcal U^\DD}_\D$.
\end{lem}
\begin{proof}
Let $\bar L$ be a $\Pi$-closure of $L$, and let $\bar K$ a $\Pi$-closure of $K$. Since $(\bar L^\DD,\D)$ is a $\D$-closure of $(L^\DD,\D)$, $\bar L^\DD$ embedds in $\bar K$. Thus, we may assume $\bar L^\DD< \bar K$. Now let $\beta$ be a solution of $(\star)$ in $\bar K$. Then, $\alpha=\beta\cdot c$ for some $c\in G(\bar L^\DD)\subset G(\bar K)$. Thus, $\alpha$ is a tuple from $\bar K$, and consequently $L<\bar K$.  

For the ``consequently'' clause, suppose $\sigma:L\to \mathcal U$ is a $\Pi$-isomorphism over $K$. Since $L$ is contained in a $\Pi$-closure of $K$, there is a formula $\phi$ (in the language of $\Pi$-rings) isolating the type $tp(\alpha/K)$. A standard argument shows that $\phi$ also isolates the type $tp(\alpha/K\gen{\mathcal U^\DD}_\D)$. This implies that 
$$tp(\alpha/K\gen{\mathcal U^\DD}_\D)=tp(\sigma(\alpha)/K\gen{\mathcal U^\DD}_\D).$$ 
It then follows (from the fact that $L\gen{\mathcal U^\DD}_\D=\sigma(L)\gen{\mathcal U^\DD}_\D$, see Remark \ref{remonex} (2)) that $\sigma$ extends uniquely to a $\Pi$-automorphism of $L\gen{\mathcal U^\DD}_\D$ over $K\gen{\mathcal U^\DD}_\D$.
\end{proof}

Let $L$ be a PSN extension of $K$ for $(\star)$. The Galois group of $\Pi$-automorphisms of $L$ is defined as
$$\operatorname{Gal}(L/K):=Aut_\Pi(L\langle \mathcal U^\DD\rangle_\Pi/K\langle \mathcal U^\DD\rangle_\Pi).$$
That is, $\operatorname{Gal}(L/K)$ is the group of $\Pi$-automorphisms of $L\langle \mathcal U^\DD\rangle_\Pi$ fixing $K\langle \mathcal U^\DD\rangle_\Pi$ pointwise. We also have the group
$$\operatorname{gal}(L/K):=Aut_\Pi(L/K).$$
By Lemma \ref{ongalgro}, $\operatorname{gal}(L/K)$ can be identified with a subgroup of $\operatorname{Gal}(L/K)$. 

We have the following important facts about the Galois group:

\begin{thm}
Suppose $L=K\gen \alpha_\Delta$ is a PSN extension for $(\star)$. Then
\begin{enumerate}
\item [(i)] There is a $\Delta$-algebraic subgroup $H$ of $G$ defined over $K^\DD$ such that the $\Pi$-algebraic group of $\DD$-constant points of $H$, $H(\mathcal U^\DD)$, is isomorphic to $\operatorname{Gal}(L/K)$. Moreover, the isomorphism 
$$\mu:\operatorname{Gal}(L/K)\to H(\mathcal U^\DD)$$
is given by $\mu(\sigma)=\alpha^{-1}\cdot \sigma(\alpha)$.
\item [(ii)] There is a natural Galois correspondence between the intermediate $\Pi$-fields (of $K$ and $L$) and the $\Delta$-algebraic subgroups of $H$ defined over $K^\DD$.
\item [(iii)] $\mu(\operatorname{gal}(L/K))=H(K^\DD).$
\end{enumerate}
\end{thm}
\begin{proof}
These are standard arguments; however, we provide details for the sake of completeness.

\noindent (i) Let $Z$ be the set of realisations (from $\mathcal U$) of the type of $\alpha$ over $K$; i.e., $Z=tp(\alpha/K)^{\mathcal U}$. Note that, by Lemma \ref{ongalgro}, $Z$ is a $K$-definable set. Set $V=\{g\in G(\mathcal U^\DD): Z\cdot g=Z\}$. Then $V$ is a $\Pi$-algebraic subgroup of $G(\mathcal U^\DD)$ defined over $K$, by Fact \ref{ConstFact}, $V$ is actually defined over $K^\DD$. The desired $\D$-algebraic group is $H:=\D$-$\operatorname{Clo}(V)\leq G$, which is defined over $K^\DD$. Clearly, by the correspondence of Lemma \ref{onsubvar} (2), $V=H(\mathcal U^\DD)$, and furthermore for every $\sigma\in \operatorname{Gal}(L/K)$ we have that $\alpha^{-1}\cdot \sigma(\alpha)\in V$. Moreover, we have
$$\mu(\sigma_1\circ\sigma_2)=\alpha^{-1}\sigma_1(\sigma_2(\alpha))=\alpha^{-1}\sigma_1(\alpha)\sigma_1(\alpha^{-1}\sigma_2(\alpha))=\alpha^{-1}\sigma_1(\alpha)\alpha^{-1}\sigma_2(\alpha)=\mu(\sigma_1)\mu(\sigma_2),$$
where the third equality uses that $\sigma_1(\alpha^{-1}\sigma_2(\alpha))=\alpha^{-1}\sigma_2(\alpha)$ which follows from the fact that $\alpha^{-1}\sigma_2(\alpha)\in G(\mathcal U^\DD)$. Thus, $\mu$ is a group homomorphism. It can easily be checked that $\mu$ is a bijection, and so a group isomorphism.

\noindent (ii) By the correspondence of Lemma \ref{onsubvar} (2), it suffices to show the Galois correspondence between intermediate $\Pi$-subfields and $\Pi$-algebraic subgroups of $V=H(\mathcal U^\DD)$ defined over $K^\DD$. The correspondence is given as follows: If $K\leq F\leq L$ is an intermediate $\Pi$-field, then $L/F$ is a PSN extension for $(\star)$, and so $V_F:=\mu(\operatorname{Gal}(L/F))$ is a $\Pi$-algebraic subgroup of $V\leq G(\mathcal U^\DD)$ deifined over $F$. Since $G(L^\DD)=G(K^\DD)$, Fact \ref{ConstFact} implies that $V_F$ is actually defined over $K^\DD$. 

Now we prove that the correspondence $F\mapsto V_F$ is 1:1 and onto. Let $F_1\neq F_2$ be intermediate $\Pi$-fields. Let $b\in F_2\setminus F_1$. Then there is $\sigma\in \operatorname{Aut}_\Pi(\mathcal U/F_1)$ such that $\sigma(b)\neq b$. Setting $\sigma'$ to be the unique extension of $\sigma |_{L}$ to an element of $\operatorname{Gal}(L/F_1)$ (see Lemma \ref{ongalgro}), we see that $\sigma' \in \operatorname{Gal}(L/F_1)\setminus \operatorname{Gal}(L/F_2)$; in particular, $V_{F_1}\neq V_{F_2}$. For surjectivity, let $W$ be a $\Pi$-algebraic subgroup of $V$  defined ove $K^\DD$. Let $b$ be a tuple from $L$ that generates the minimal $\Pi$-field of definition of $\alpha\cdot W$. If we let $F=K\langle b\rangle_\Pi$, it can be checked that $V_F=W$. 

\noindent (iii) Let $\sigma\in \operatorname{Gal}(L/K)$. If $\mu(\sigma)\in H(K^\DD)$, then $\sigma(\alpha)=\alpha\cdot c$ for some $c \in H(K^\DD)$. In particular, $\sigma(\alpha)$ is a tuple from $L$, and so $\sigma\in \operatorname{gal}(L/K)$. On the other hand, if $\sigma\in\operatorname{gal}(L/K)$, then $\alpha^{-1}\sigma(\alpha)\in H(L^\DD)=H(K^\DD)$.
\end{proof}

Let us now discuss the issue of \emph{existence} of PSN extensions (they do not generally exist as we pointed out in the introduction). In \cite[\S5]{Omar1}, it was shown that these extensions exist if we make the additional assumption that $(K^\DD,\Delta)$ is $\D$-closed. To see this, let $\bar K$ be a $\Pi$-closure of $K$. Recall that any tuple of $\DD$-constants from $\bar K$ is $\Delta$-constrained (or equivalently $\Delta$-isolated) over $K^\DD$ and so, by the assumption on $(K^\DD,\Delta)$, this tuple must be in $K^\DD$. Now, since $a$ is a $K$-point in the image of $\ell_0$, we can find a solution $\alpha$  of $(\star)$ in $\bar K$. Setting $L:=K\gen{\alpha}_\Delta$ and using $\bar K^\DD=K^\DD$, we have that $L^\DD=K^\DD$ and so $L$ is a PSN extension of $K$ for $(\star)$. Moreover, the assumption that $(K^\DD,\Delta)$ is $\Delta$-closed implies that, up to $\Pi$-isomorphism over $K$, this is the only PSN extension of $K$ for $(\star)$.

There are weaker assumptions on the field of $\DD$-constants $K^\DD$ that imply the existence of PSN extensions. For instance, we have the following result of Wibmer:

\begin{thm}\label{wib}\cite[Theorem 8]{Wibmer}
When $G=\operatorname{GL}_n$ and $\D=\{\delta\}$, PSN extensions exist if $K^\DD$ is algebraically closed.
\end{thm}

It should be noted that Wibmer works in the $\delta$-parameterized Picard-Vessiot context with systems of linear \emph{difference-differential} equations. Thus, Theorem \ref{wib} is a special case (when the set of automorphisms is empty) of the main result of \cite{Wibmer}.

Our next goal is to extend Theorem \ref{wib} to an arbitrary set of parametric derivations $\Delta$, and $G$ not necessarily linear. To do so, we will need the following result, see \cite[Chapter 0, \S7]{Kolchin2}.

\begin{fct}\label{extKol}
Suppose $(F,\bar \partial\cup\{\delta\})$ is a differential field with $\bar\partial =\{\partial_1,\dots,\partial_s\}$ (in particular we require that all the derivations commute). If $(E,\bar\partial)$ is a differential field extension of $(F,\bar\partial)$ which is $\bar\partial$-closed, then there exists an extension $\delta':E\to E$ of $\delta$ such that $(E,\bar\partial\cup\{\delta'\})$ is a differential field (i.e., $\delta'$ commutes with $\bar\partial$).  
\end{fct}

\begin{rem}
It is not known (at least to the authors) if the above result holds if we replace $\delta$ for a finite family of derivations commuting with each other and with $\bar\partial$ on $F$.
\end{rem}



If $\Delta\neq \emptyset$, we set 
$$\Delta^*=\Delta\setminus\{\delta_1\}.$$
The following theorem gives sufficient conditions for the existence of parameterized stronly normal extensions.

\begin{thm}\label{genwib}
Suppose $G$ is a $\Delta^*$-algebraic group defined over $K^\DD$. If $(K^\DD,\Delta^*)$ is $\Delta^*$-closed, then there exists a parameterized strongly normal extension of $K$ for $(\star)$.
\end{thm}
\begin{proof}
Let $\bar K$ be a $\DD\cup\Delta^*$-closure of $K$. Any tuple of $\DD$-constants from $\bar K$ is $\Delta^*$-constrained (or equivalently $\Delta^*$-isolated) over $K^\DD$ and so, by the assumption on $(K^\DD,\Delta^*)$, this tuple must be in $K^\DD$. In particular, $\bar K^\DD=K^\DD$. Now, by Fact~\ref{extKol}, there is an extension $\delta_1':\bar K\to \bar K$ of $\delta_1:K\to K$. This yields a differential field extension $(\bar K, \DD\cup\Delta')$ of $(K,\DD\cup\Delta)$ where $\Delta'=\{\delta_1'\}\cup\Delta^*$. 

Since $G$ is a $\Delta^*$-algebraic group, we can find a $\bar K$-point $\alpha\in G$ such that $\ell_0(\alpha)=a$. Let $L$ be the differential subfield of $(\bar K, \DD\cup\Delta')$ generated by $\alpha$ over $K$; in other words, $L=K\langle \alpha\rangle_{\Delta'}$. This is the desired PSN extension. Indeed, it is $\Delta'$-generated by a solution of $(\star)$ and $L^\DD$ is contained in $\bar K^\DD=K^\DD$.
\end{proof}

\begin{rem}\label{remonwib}\
\begin{enumerate}
\item When $G$ is a $\Delta^*$-algebraic group and $\Delta=\{\delta\}$, the above theorem shows that in order to guarantee the existence of a parameterized strongly normal extension, it suffices to assume that $K^\DD$ is algebraically closed.
\item When $G=\operatorname{GL}_n$ and $\Delta=\{\delta\}$, the above theorem shows that if $K^\DD$ is algebraically closed, then parameterized Picard-Vessiot extensions of $K$ exist. This is Theorem \ref{wib} above and, as we already mentioned, is (a special case of) the main result of \cite{Wibmer}.
\end{enumerate}
\end{rem}

The remainder of the paper is devoted to prove the existence of PSN extensions under the assumption that $(K^\DD,\D)$ is existentially closed in $(K,\D)$. We achieve this by extending some of the results of Kamensky and Pillay from \cite{MoshePillay}.

\section{A quantifier elimination result for parameterized D-groups}\label{QED}

The goal of this section is to proof a quantifier elmination result for certain theories of parameterized D-groups (cf. \cite[\S2]{MoshePillay}). This result will be used in the next section. We continue with the notation of the previous sections, and again fix $\Delta$-algebraic group $G$ defined over $K^{\mathcal{D}}$ and an integrable $K$-point $a$ of $(G,s_0)$. Recall that the latter means that $(G,s)$ is a parameterized D-variety, where $s:=as_0$. 

We fix the parameterized logarithmic equation
\[\ell_{0}(y)=a, \quad \text{ where } y \text{ ranges in } G \tag{$\star$},\]
and let $\mathcal{Y}=\{y\in G: \ell_{0}(y)=a\}$ be its solution set. Note that 
$$\mathcal Y=(G,s)^\#=b\cdot G(\mathcal U^\DD)=b \cdot \operatorname{Ker}(\ell_0)$$
for any $b\in \mathcal Y$.

For each $n,m\geq 0$, the product $\mathcal U^n\times  G^m$ has a natural structure of parameterized D-variety induced from the zero section on $\mathcal U$ and the section $s=as_0$ on $G$. By a parameterized D-subvariety of $\mathcal U^n\times G^m$ we mean one with respect to this parameterized D-structure. In particular, the set of sharp points of $\mathcal U^n\times G^m$ is $(\mathcal U^\DD)^n\times \mathcal Y^m$, and, by Lemma \ref{onsubvar} (2), a $\D$-subvariety $W$ of $\mathcal U^n\times G^m$ is a parameterized D-subvariety iff $W^\#:=W\cap \left((\mathcal U^\DD)^n\times \mathcal Y^m\right)$ is $\D$-dense in $W$.

The two-sorted language $\mathcal L_{D,K}$ has symbols $R_W$ for each parameterized D-subvariety $W$ of $\mathcal U^n\times G^m$ defined over $K$. We denote by $(\mathcal U, G)_{D,K}$ the $\mathcal L_{D,K}$-structure with sorts $\mathcal U$ and $G$ where the interpretation of $R_W$ is the tautological one, namely $W$. On the other hand, $(\mathcal U^\DD, \mathcal  Y)_{D,K}$ denotes the $\mathcal L_{D,K}$-structure with sorts $\mathcal U^\DD$ and $\mathcal Y$ where the interpretation of $R_W$ is $W^\#$.

\begin{rem}\label{propi} \
\begin{enumerate}
\item The parameterized D-subvarieties of $\mathcal U$ (with its zero section) and of $(G,s_0)$ are precisely those $\D$-subvarieties of $\mathcal U$ and $G$, respectively, that are defined over $\mathcal U^\DD$. A similar observation holds for parameterized D-subvarieties defined over $K$.
\item By quantifier elimination of $DCF_{0,m}$ and Lemma \ref{onsubvar} (2), every $\emptyset$-definable set of the structure $(\mathcal U^\DD,\mathcal Y)_{D,K}$ is a (finite) Boolean combination of $W^\#$'s where the $W$'s are parameterized D-subvarieties of some $\mathcal U^n\times G^m$ defined over $K$. In particular, the structure $(\mathcal U^\DD,\mathcal Y)_{D,K}$ has quantifier elimination.
\end{enumerate}
\end{rem}

Here is the main result of this section (cf. \cite[Corollary 2.5]{MoshePillay}).

\begin{thm}\label{QEthm}
The structure $(\mathcal{U},G)_{D,K}$ has quantifier elimination.
\end{thm}
\begin{proof} 
This is an extension of the arguments of \cite[Lemma 2.4 and Corollary 2.5]{MoshePillay}. Let $A\subseteq \mathcal U^n\times G^m$ be any constructible set of the structure $(\mathcal U, G)_{D,K}$ and $\pi:\mathcal U^n\times G\to \mathcal U^{n'}\times G^{m'} $ a coordinate-projection map. We must show that $\pi(A)$ is again constructible; i.e., that $\pi(A)$ is a (finite) Boolean combination of parameterized D-subvarieties of $\mathcal U^{n'}\times G^{m'}$ which are defined over $K$. 

We first check that $\pi(A)$ is a Boolean combination of parameterized D-varieties (defined over $\mathcal U$). Fix $b\in \mathcal Y$. Let $\lambda:G\to G$ be left-translation by $b^{-1}$. Since $\mathcal Y=b\cdot G(\mathcal U^\DD)$, the map $\lambda$ induces a bijection between $\mathcal Y$ and $G(\mathcal U^\DD)$. By Proposition \ref{onsubvar} (2) and Remark \ref{propi} (1), a $\D$-subvariety $B$ of $\mathcal U^n\times G^m$ is a parameterized D-subvariety iff the $\D$-subvariety $B':=(\operatorname{Id}^n, \lambda^m)(B)$ is defined over $\mathcal U^\DD$. Consequently, $A':=(\operatorname{Id}^n,\lambda^m)(A)$ is a Boolean combination of $\D$-subvarieties of $\mathcal U^n\times G^m$ defined over $\mathcal  U^\DD$. By quantifier elimination of $DCF_{0,m-r}$ (applied in the structure $(\mathcal U, \D)$), $\pi(A')$ is a Boolean combination of $\D$-subvarieties of $\mathcal U^{n'}\times G^{m'}$ defined over $U^{\DD}$. Applying the inverse of $(\operatorname{Id}^{n'},\lambda^{m'})$, we obtain that $\pi(A)$ is a Boolean combination parameterized D-subvarieties of $\mathcal U^{n'}\times G^{m'}$ defined over $\mathcal U$.

Finally, since $\pi(A)$ is definable over $K$, Lemma \ref{usil} implies that $\pi(A)$ is a Boolean combination of parameterized D-subvarieties all defined over $K$.
\end{proof}

We end this section with an easy consequence of the above theorem (and Proposition \ref{uselater}). 

\begin{cor}\label{Interpretation}
$(\mathcal{U}^{\mathcal{D}},\mathcal{Y})_{D,K}$ is an elementary substructure of $(\mathcal{U},G)_{D,K}$. Consequently, the map assigning to each symbol $R_W \in L_{D,K}$ the formula (over $K$) defining $W$ in $(\mathcal{U},\Delta)$ is an interpretation of $Th(\mathcal{U}^{\mathcal{D}},\mathcal{Y})_{D,K}$ in $DCF_{0,m-r,K}$.
\end{cor}
\begin{proof}
Let $A$ be a nonempty definable set of $(\mathcal U,G)_{D,K}$. By the usual Tarski-Vaught test, it suffices to show that $A$ has a point in $(\mathcal U^\DD, \mathcal Y)_{D,K}$. By Theorem \ref{QEthm}, $A$ is a (finite) Boolean combination of parameterized D-subvarieties of some $\mathcal U^n\times G^m$. But, by Proposition \ref{uselater}, $A^\#$ is dense in $A$; in particular, $A^\#\subset (\mathcal U^\DD)^n \times \mathcal Y^m $ is nonempty. The result follows.
\end{proof}

\begin{rem}
It should be noted that in the ODE case algebraic D-groups (not necessarily defined over the constants) have quantifier elimination \cite{KoPi}. This result has an immediate extension to the PDE case when the parametric set of derivations $\D$ is empty (these are the algebraic D-groups studied by Buium \cite{Bu}). However, in general, $\D\neq \emptyset$, the situation is quite different. For instance, the arguments from \cite[\S2]{KoPi} rely heavily on the properties of Jet spaces and Grassmanians for algebraic varieties, and at the moment it is unclear how such properties will translate for $\Delta$-algebraic varieties. We leave the question of quantifier elimination of parameterized D-groups for future work.
\end{rem}

\section{Existence of parameterized strongly normal extensions}
In this section we prove our main result concerning the existence of parameterized strongly normal extensions. Our assumptions are similar to those in the previous section: $K$ denotes our ground $\Pi$-field and $G$ is a $\Delta$-algebraic group defined over $K^{\mathcal{D}}$. We fix an integrable $K$-point $a$ of $(G,s_0)$ and a parameterized logarithmic equation
\[\ell_{0}(y)=a, \quad \text{ where } y \text{ ranges in } G. \tag{$\star$}\]
We let $\mathcal{Y}=\{y\in G: \ell_{0}(y)=a\}$ be the solution set of $(\star)$. 

We start with some key model theoretic results and, as we shall see, most follow by adapting the proof of similar results found in \cite[\S3]{MoshePillay} and using the results of previous sections (primarily Corollary \ref{Interpretation}).
 
\begin{lem}{\color{white}kn}\label{onty}
\begin{enumerate}
\item $\mathcal{Y}$ is contained in any $\Pi$-closure of $K\gen {\mathcal{U}^{\mathcal{D}}}_\D$. In particular, for all $\alpha\in\mathcal{Y}$, $tp(\alpha/K\gen {\mathcal{U}^{\mathcal{D}}}_\D)$ is isolated.
\item For each complete type $p(x)$ over $K$ containing ``$x\in\mathcal{Y}$", there is a formula $\phi_p(x,y)$ and a definable partial function $f_p(x)$, both defined over $K$ in the language of $\Pi$-rings, such that for each realisation $\alpha$ of $p$, $f_p(\alpha)$ is a tuple from $\mathcal{U}^{\mathcal{D}}$ and $\phi_p(x,f_p(\alpha))$ isolates $tp(\alpha/K\gen {\mathcal{U}^{\mathcal{D}}}_\D)$.
\end{enumerate}
\end{lem}
\begin{proof} \

\noindent (1) This is clear.

\noindent (2) For any $\alpha\in\mathcal{Y}$, by elimination of imaginaries in $DCF_{0,m}$, we may assume that the formula $\phi(x,c)$ that isolates $tp(\alpha/K\gen {\mathcal{U}^{\mathcal{D}}}_\D)$ is such that $c$ (a tuple from $\mathcal{U}^{\mathcal{D}}$) is a canonical parameter for $\phi(x,c)$ over $K$. So we have a $K$-definable partial function $f_p$, depending only on $p=tp(\alpha/K)$, so that $c=f_p(\alpha)$. As the formula $\phi(x,y)$ also depends only on $p$, we write it as $\phi_p(x,y)$.
\end{proof}

One of the main outcomes of the above lemma is the following:

\begin{lem}\label{Const} 
Let $p(x)$ be a complete type over $K$ containing ``$x\in\mathcal{Y}$". Suppose $c=f_p(\alpha)$ for some $\alpha$ realizing $p$, where $f_p$ is as given in Lemma~\ref{onty}. Then $(K\gen{\alpha}_{\Delta})^{\mathcal{D}}=K^{\mathcal{D}}\gen{c}_{\Delta}$.
\end{lem}
\begin{proof}
Since $c$ is a tuple from $(K\gen{\alpha}_{\Delta})^{\mathcal{D}}$ all we really have to show is that $(K\gen{\alpha}_{\Delta})^{\mathcal{D}}\subseteq K^{\mathcal{D}}\gen{c}_{\Delta}$. So let $\beta\in (K\gen{\alpha}_{\Delta})^{\mathcal{D}}$, say $\beta=g(\alpha)$ for some $K$-definable partial function $g$. But $``\beta=g(x)"\in tp(\alpha/K\gen {\mathcal{U}^{\mathcal{D}}}_\D)$ and so we have $\mathcal U \models \forall x(\phi_p(x,c)\rightarrow \beta=g(x))$. Hence $\beta$ is a tuple from $K\gen{c}_\D$. Using Fact \ref{ConstFact}, we see that $\beta$ is in fact a tuple from $K^{\mathcal{D}}\gen{c}_\D$. The result follows.
\end{proof}

Note that Lemma \ref{Const} indicates that one way to prove that a PSN extension of $K$ for $(\star)$ exists is to find an $\alpha\in\mathcal{Y}$ such that $f_p(\alpha)\in K^{\mathcal{D}}$, where $p=tp(\alpha/K)$. This is precisely what we aim to do under the assumption that the $\D$-field $(K^{\mathcal{D}},\D)$ is existentially closed in $(K,\D)$. First we need to prove that we can find a single $f$ and $\phi$ to do the job of the $f_p$'s and $\phi_p$'s:

\begin{prop}\label{mainPSN}
There is a formula $\phi(x,y)$ over $K$ and $K$-definable function $f:\mathcal{Y}\rightarrow (\mathcal{U}^{\mathcal{D}})^m$ for some $m$ such that for all $\alpha\in\mathcal{Y}$ the formula $\phi(x,f(\alpha))$ isolates $tp(\alpha/K\gen {\mathcal{U}^{\mathcal{D}}}_\D)$.
\end{prop}
\begin{proof} We prove this by means of two claims. We follow closely the strategy of \cite[\S3]{MoshePillay}.

\vspace{.05in}
\noindent {\bf Claim 1.} Let $\alpha\in\mathcal{Y}$ and $\sigma\in Aut(\mathcal{Y}/K\gen {\mathcal{U}^{\mathcal{D}}}_\D)$. Then $\sigma(\alpha)\cdot\alpha^{-1}$ does not depend on $\alpha$.\\
{\it Proof of claim.} If $\beta$ is another element of $\mathcal{Y}$, then $\alpha=\beta\cdot c$ for some $c\in G(\mathcal{U}^{\mathcal{D}})$. So from $\sigma(\alpha)=\sigma(\beta)\cdot c$ and $c=\beta^{-1}\cdot\alpha$, we have that $\sigma(\alpha)\cdot\alpha^{-1}=\sigma(\beta)\cdot\beta^{-1}$.

\vspace{.05in}
\noindent {\bf Claim 2:} The map $\rho:Aut(\mathcal{Y}/K\gen {\mathcal{U}^{\mathcal{D}}}_\D)\rightarrow G$ taking $\sigma$ to $\sigma(\alpha)\cdot\alpha^{-1}$ (for any $\alpha\in \mathcal{Y}$) is an isomorphism between $Aut(\mathcal{Y}/K\gen {\mathcal{U}^{\mathcal{D}}}_\D)$ and a $K$-definable subgroup $H^+$ of $G$.\\
{\it Proof of claim.} For any $\alpha\in \mathcal{Y}$, letting $p=tp(\alpha/K)$, $H^+$ is definable over $K\gen{f_p(\alpha)}_{\Delta}$ as 
$$\{x\cdot\alpha^{-1}:\;  \mathcal U\models\phi_p(x,f_p(\alpha))\}.$$
Since $H^+$ does not depend on the choice of $\alpha$ it is defined over $K$.

To finish the prove, let $Z$ be the set $\mathcal{Y}/H^+$ of right cosets of $H^+$ in $\mathcal{Y}$. By elimination of imaginaries we have a $K$-definable $f:\mathcal{Y}\rightarrow \mathcal{U}^m$ (for some $m$) such that $Z$ can be considered as the definable set $f(\mathcal{Y})$. Moreover, as $Z$ is fixed pointwise by $Aut(\mathcal{U}/K\gen {\mathcal{U}^{\mathcal{D}}}_\D)$, we have that our $K$-definable function is in fact a map from  $\mathcal{Y}$ to $(\mathcal{U}^{\mathcal{D}})^m$.
\end{proof}

Recall that a $\D$-subfield $F$ of $K$ is said to be \emph{existentially closed in} $(K,\D)$ if every quantifier-free formula in the language of $\D$-rings over $F$ with a solution in $K$ has a solution in $F$. It is easy to see that $(F,\D)$ is existentially closed in $(K,\D)$ if and only if any $\D$-algebraic variety over $F$ with a $K$-point has an $F$-point. 

We are now ready to prove:

\begin{thm}
Suppose $(K^{\mathcal{D}},\D)$ is existentially closed in $(K,\D)$. Then there is a parameterized strongly normal extension $L$ of $K$ for the equation $(\star)$.
\end{thm}
\begin{proof}


The function $f$ given in Proposition \ref{mainPSN} and its image $Im(f)$ are $\emptyset$-definable in the structure $(\mathcal{U}^{\mathcal{D}},\mathcal{Y})_{L_{D,K}}$ introduced in Section 4. Moreover, Fact \ref{ConstFact} tell us that $Im(f)$ is definable over $K^{\mathcal{D}}$, by a formula $\chi(y)$ say.

We use the interpretation result (Corollary \ref{Interpretation}) in two ways. First we take $F$ to be the interpretation of $f$ in $(\mathcal{U},G)$. So $F$ is a function from $G$ to $\mathcal{U}^m$ definable over $K$ in the language of $\D$-rings; i.e., $K$-definable in the structure $(\mathcal{U},\D)$. We also take $\Gamma$ to be the interpretation of $\chi$ in $(\mathcal{U},G)$. We have that $\Gamma(y)$ defines $Im(F)$ and is also definable over $K^{\mathcal{D}}$ in $(\mathcal{U},\D)$.

Now, as the identity element $e$ of $G$ is a $K$-point, we have that $F(e)\in K^m$ and so $Im(F)$ has a point in $K$. In other words, $\Gamma(y)$ is realized in $K$. By the assumption that $(K^{\mathcal{D}},\Delta)$ is existentially closed in $(K,\Delta)$, $\Gamma(y)$ is realized in $K^{\mathcal{D}}$, say by $a$. So ``there is an $x$ such that $F(x)=a$" is true in the structure  $(\mathcal{U},G)_{D,K}$. Because $(\mathcal{U}^{\mathcal{D}},\mathcal{Y})_{D,K}$ is an elementary substructure of $(\mathcal{U},G)_{D,K}$ and $a\in K^{\mathcal{D}}$, one can find such an $x$ in $\mathcal{Y}$.

So  there is an $\alpha\in\mathcal{Y}$ such that $f(\alpha)\in (K^{\mathcal{D}})^m$. By Lemma \ref{Const},  $(K\gen{\alpha}_{\Delta})^{\mathcal{D}}=K^{\mathcal{D}}\gen{f(\alpha)}_{\Delta}=K^{\mathcal{D}}$. Hence $L=K\gen{\alpha}_{\Delta}$ is a parametrized strongly normal extension of $K$ for $(\star)$.\\

\end{proof}


\end{document}